\documentclass[11pt]{amsart}   
\usepackage{latexsym}
\usepackage{amsfonts}
\usepackage{amscd}
\usepackage[all]{xypic}
\theoremstyle{plain}

\newtheorem{definition}{Definition}

\newtheorem{lemma}{Lemma}

\newtheorem{proposition}{Proposition}

\newtheorem{theorem}{Theorem}
\numberwithin{equation}{section}

\begin{document}


\vspace{0.5in}

\title[Bohr Compactification
and Almost Periodic Functions] {Extension of continuous functions
on product spaces, Bohr Compactification and Almost Periodic
Functions}
\author{Salvador Hern\'andez}
\address{Universitat Jaume I, Departamento de Matem\'{a}ticas,
Campus de Riu Sec, 12071-Castell\'{o}n, Spain}
\email{hernande@mat.uji.es} \dedicatory{Dedicated to Professor Wis
Comfort}
\thanks{Research partially supported by Spanish DGES, grant BFM2000-0913,
and Generalitat Valenciana, grant CTIDIB/2002/192}
\date{}

\begin{abstract}
The Bohr compactification is a well known construction for
(topological) groups and semigroups. Recently, this notion has
been investigated for arbitrary structures in
\cite{har_kun:bohr_discrete} where the Bohr compactification is
defined, using a set-theoretical approach, as the maximal
compactification which is compatible with the structure involved.
Here, we give a characterization of the continuous functions
defined on a product space that can be extended continuously to
certain compactifications of the product space. As a consequence,
the Bohr compactification of an arbitrary topological structure is
obtained as the Gelfand space of the commutative Banach algebra of
all almost periodic functions. Previously, almost periodic
functions $f$ are defined in terms  of translates of $f$ with no
reference to any compactification of the underlying structure. An
application is given to the representation of isometries defined
between spaces of almost periodic functions.
\end{abstract}

\maketitle

\section{Introduccion}
Hart and Kunen in \cite{har_kun:bohr_discrete} have defined and
investigated the Bohr compactification and topology of an
arbitrary discrete structure (see also \cite{holm}). Even though
their approach to the question is set-theoretical, they also
comment that the name of Bohr is attached to this construction
stems from the fact that, for topological groups, this compact
structure can be defined via almost periodic functions that were
introduced by Harald Bohr. In fact, in
\cite{har_kun:bohr_discrete} is also implicit the question of
defining the Bohr compactification of arbitrary structures using
"appropriately" defined almost periodic functions (see
\cite[2.3.12]{har_kun:bohr_discrete}). The motivation for this
paper is to study this question. We define almost periodic
functions $f$ directly in terms of the translates of $f$ without
any previous reference to any compactification of the underlying
structure. We also characterize the continuous functions defined
on a product space that can be extended to certain
compactifications of the product space. As a consequence, it is
proved that the Bohr compactification introduced in
\cite{har_kun:bohr_discrete} is canonically equivalent to the
Gelfand or structure space associated to the commutative Banach
algebra of all almost periodic functions. In
\cite{prod1,prod2,prod3,prod4} I. Prodanov established a theory of
almost periodic function for certain general topological
structures that he called {\em continous algebraic structures}.
Prodanov's results are related to the ones in this paper and, in
fact, our definition of almost periodic function is similar to the
one considered by him. Thus, part of our results can also be
obtained using Prodanov's approach.

The basic definitions and terminology are taken from
\cite{har_kun:bohr_discrete,engel,roelcke_dierolf,rudin_fa,semadeni}.
Firstly, we recall some basic facts that will be used along the
paper.

Let $X$ be a set and $\Delta= \Delta(X)\subset X\times X$ the
diagonal on $X$. For $B,C\subset X\times X$, $S\subset X$ we
define $B\circ C= \{(x,z)\in X\times X : (x,y)\in B \ \hbox{and} \
(y,z)\in C \ \hbox{for some} \ y\in X \}$, $B^{-1}= \{(x,y) :
(y,x)\in B \}$, $B[S]=\{y\in X : (x,y)\in B \ \hbox{for some} \
x\in X \}$. A {\it uniformity} on $X$ is a set $\mu$ of subsets of
$X\times X$ satisfying the following conditions: (i) $B\supset
\Delta$ for all $B\in \mu$; (ii) if $B\in \mu$, then $B^{-1}\in
\mu$; (iii) if $B\in \mu$, there exists $C\in \mu$ such that
$C\circ C\subset B$; (iv) the intersection of two members of $\mu$
also belongs to $\mu$; (v) any subset of $X\times X$, which
contains a member of $\mu$, itself belongs to $\mu$.

The members of $\mu$ are called {\it vicinities} (of $\Delta$). By
a {\it base} for a uniformity $\mu$ is meant a subset $\mathcal B$
of $\mu$ such that a subset of $X\times X$ belongs to $\mu$ if and
only if it contains a set belonging to $\mathcal B$. A {\it
uniform space} $\mu X$ is a pair comprising a set $X$ and a
uniformity $\mu$ on $X$. If $\mu X$ is a uniform space, one may
define a topology $\tau$ on $X$ by assigning to each point $x$ of
$X$ the neighborhood base comprised of the sets $B[x]$, $B$
ranging over the uniformity (see \cite{roelcke_dierolf}).

The most basic properties of Banach algebras can be found in
\cite{rudin_fa}. It suffices to say here that every commutative
Banach algebra with unity  $\mathcal A$ has associated a compact
space $K=K(\mathcal A)$ ({\it the Gelfand space}), which consists
of all non null complex homomorphisms of $\mathcal A$. The algebra
$\mathcal A$ is isomorphic to a subalgebra of $C(K)$ by means of a
map $f\longrightarrow \hat{f}$, given by $\hat{f}(\chi)=\chi(f)$
for all $\chi\in K$. We call $\hat{f}$ the {\it Gelfand transform}
of $f$. If $\parallel \cdot \parallel_{\infty}$ denotes the norm
of uniform convergence on $K$, it holds that $\parallel \hat{f}
\parallel_{\infty} \leq
\parallel f \parallel$ for all $f\in \mathcal A$.

In what follows $\mathcal{L}$ is a set (possibly empty) of symbols
of constants and symbols of functions; every function symbol has
arity $\geq 1$. Using the symbols of $\mathcal{L}$ and the
predicate "$=$" one may construct logical formulae in the usual
way. Only the predicate "$=$" will be used here. A
\emph{structure} $\mathcal{U}$ for $\mathcal{L}$ is a non empty
set $A$ (the domain) together with elements (of) and functions
(defined on) $A$ corresponding to the symbols in $\mathcal{L}$.
E.g., when we talk about groups, it is understood that
$\mathcal{L}=\{\cdot ,i,1\}$ (symbols of the product, inverse
element and identity). Thus, groups are displayed as
$\mathcal{U}=(A;\cdot ,i,1)$.

Let $\mathcal{U}$ \ be a structure for $\mathcal{L}$ and
$f:A\rightarrow X$.
If $\Phi \in \mathcal{L}$ is an $n$-ary function symbol, then $f(\Phi _{%
\mathcal{U}})$ denotes
$\{(f(a_{1},...,f(a_{n}),f(b)):(a_{1},...,a_{n},b)\in \Phi
_{\mathcal{U}}\}$.
Here $\Phi _{\mathcal{U}}$ is identified to the graph of $\Phi $.
We have that $f(\Phi _{\mathcal{U}})\subset X^{n+1}$ but is not
necessary the graph of an $n$-ary function.

A \emph{topological structure} for $\mathcal{L}$ is a pair
$(\mathcal{U},\tau )$ where $\mathcal{U}$ is a structure for
$\mathcal{L}$, and $\tau $ is a topology on $A$ making all
functions in $\mathcal{U}$ continuous. We write $\mathcal{U}$ for
$(\mathcal{U},\tau )$ if the topology is understood.

Let $\mathcal{U}$ and $\mathcal{V}$ be two topological structures
of $\mathcal{L}$, and $f:A\rightarrow B$. The map $f$ is a
\emph{homomorphism} from $\mathcal{U}$ to $\mathcal{V}$ iff $f$ is
continuous, $f(\Phi _{\mathcal{U}})\subset \Phi _{\mathcal{V}}$
for each function symbol $\Phi $ of $\mathcal{L}$, and
$f(c_{\mathcal{U}})=c_{\mathcal{V}}$ for each constant symbol $c$
of $\mathcal{L}$.

A \emph{compact structure} for $\mathcal{L}$ is a topological
structure $(\mathcal U,\tau)$ in which $\tau$ is a compact
Hausdorff topology.

Let $A$ be any non-empty set. A \emph{compactification} of $A$ is
a pair $(X,\varphi)$, where $X$ is a compact space,
$\varphi:A\longrightarrow X$, and $\varphi(A)$ is dense in $X$. If
$(X,\varphi)$ and $(Y,\psi)$ are two compactifications of $A$,
then $(X,\varphi)\leq_{\Gamma} (Y,\psi)$ means that
$\Gamma:Y\longrightarrow X$ is a continuous function and
$\Gamma\circ \psi=\varphi$. $(X,\varphi)\leq (Y,\psi)$ means that
$(X,\varphi)\leq_{\Gamma} (Y,\psi)$ for some $\Gamma$.

If $(\mathcal U, \tau)$ is a topological structure and
$(X,\varphi)$ is a compactification of the set $A$, then
$(X,\varphi)$ is \emph{compatible} with $(\mathcal U, \tau)$ iff
$\varphi$ is continuous and there is a topological structure
$\mathcal X$ built on the set $X$ such that $\varphi$ is a
homomorphism.

The \emph{Bohr compactification}, $(b(\mathcal
U,\tau),\Phi_{(\mathcal U,\tau)})$, of a given topological
structure $(\mathcal U,\tau)$, is the maximal compatible
compactification. The $\tau$ is omitted when it is clear from
context.

Here on, we consider a topological structure $(\mathcal{U},\tau )$
for an arbitrary but fixed set $\mathcal{L}$ of constants and
functions. If $A$ is the set underlying the structure
$\mathcal{U}$, we have that $(A,\tau )$ is a topological space
such that for each $\Phi \in \mathcal{L}$, the map $\Phi
:A^{n}\rightarrow A$ is continuous on $\ A^{n}$, here "$n$" is the arity of $%
\Phi $. In order to simplify the notation in the above situation
later on,
we consider the following symbolism: for $\Phi \in \mathcal{L}$ of arity "$n$%
", let $A_{i}=A$ for $i=1,...,n$ and consider $\Phi :A_{1}\times
A_{2}\times ...\times A_{n}\rightarrow A$ canonically defined.
Thus, with some
notational abuse, each $(a_{i_{1}},a_{i_{2}},...,a_{i_{m}})\in $ $%
A_{i_{1}}\times A_{i_{2}}\times ...\times A_{i_{m}}$, $1\leq m\leq
n$, defines the map $\Phi
_{(a_{i_{1}},a_{i_{2}},...,a_{i_{m}})}:A_{j_{1}}\times
A_{j_{2}}\times ...\times A_{j_{(n-m)}}\rightarrow A$ by $\Phi
_{(a_{i_{1}},a_{i_{2}},...,a_{i_{m}})}(a_{j_{1}},...,a_{j_{(n-m)}})=\Phi
(a_{1},a_{2},...,a_{n})$. Using this symbolism, if we take any $j$
with $1\leq j\leq n$ and an arbitrary but fixed
$\overrightarrow{x}\in \prod \{A_{i} : 1\leq i\leq n\hbox{,}\
i\not= j\}$ , we define a \textit{translation}
$t_{\overrightarrow{x}}^{\Phi }$ on $A$ ($t_{\overrightarrow{x}}$
for short if there is no possible confusion) by the rule
$t_{\overrightarrow{x}}(a)= \Phi_{\overrightarrow{x}}(a)=\Phi
(\overrightarrow{x};a)$, here on the symbol "$;$" is used to mean
that the \textit{variable} $a$ is placed at the coordinate "$j$".
We say that $t_{\overrightarrow{x}}$ is a \emph{simple
translation} on $A$. In case $\Phi $ is $1$-ary, we define the translation $%
t_{\emptyset}^{\Phi }$ to be $\Phi $. Simple translations can be
multiplied using the ordinary composition of mappings. Thus, the
set of all simple translations generates the semigroup of
(general) translations $S(\mathcal{U})$ with the composition law
defined by ordinary function composition.  To emphasize the fact
that simple translations are defined on $A$, we shall use the
symbol $t_{\overrightarrow{x}}^{\Phi }$ when
$\overrightarrow{x}\in \prod \{A_{i} : 1\leq i\leq n\hbox{,}\
i\not= j\}, \ 1\leq j\leq n$. Otherwise, we shall use the most
standard symbol $\Phi_{\overrightarrow{x}}$. In the sequel, if $X$
is a topological space, we denote by $C_{\infty}(X)$ the set of
all complex valued continuous functions on $X$ equipped with the
suppremum norm.

\begin{definition}
\label{def3}The map $f:A\rightarrow \mathbb{C}$ is said to be
almost periodic when is bounded, continuous and it holds that the
set $\{(f\circ \tau\circ \Phi_{a_{j}}):a_{j}\in A\}$ is relatively
compact in $C_{\infty}(\prod \{A_{i} : 1\leq i\leq n\hbox{,}\
i\not= j\})$ for all $\Phi \in \mathcal{L}$, all $\tau\in
S(\mathcal{U})$, and all $j$ with $1\leq j\leq n$, with $n$ being
the arity of $\Phi $.
\end{definition}

As a consequence of the definition above, for all $f$ almost periodic and $%
\tau\in S(\mathcal{U})$, it holds that $f\circ \tau$ is almost
periodic. The set of all almost periodic functions on a
topological structure $\mathcal{U}$ is denoted by
$AP(\mathcal{U})$. Obviously, when the composition of any two
simple translations yields a simple translation, we have that the
set of all simple translations is itself a semigroup. This
happens, for example, for groups and semigroups.

In principle, our approach for defining the Bohr compactification
of an arbitrary structure generalize the one given by Loomis
\cite{loomis} and Semadeni \cite{semadeni} for topological groups.
Nevertheless, the obvious complication that arise when we want to
extend the operations of a given algebraic structure to its Bohr
compactification stems from the fact that there are {\it many}
arbitrary operations of different arity in general. This means
that the set of simple translations is far from being a semigroup
and, as a consequence, the usual proofs given for groups and
semigroups do not work here. Thus, our approach is also
topological since it is based on Proposition \ref{prop3} below,
which is a result about extension of continuous functions defined
on a compact space.

\section{Main results}
The goal now is to study the properties of $AP(\mathcal{U})$. We
want $AP(\mathcal{U})$ to be a commutative Banach algebra so that
its structure or Gelfand space be isomorphic \ to $b\mathcal{U}$
as it was defined in \cite{har_kun:bohr_discrete}. The proof of
the proposition below is more or less standard, we include part of
it for the reader's sake (see \cite[\S 14.7] {semadeni}).

\begin{proposition}
\label{prop1}The set $AP(\mathcal{U})$ is a closed subalgebra of
$C_{\infty }(A)$ containing the constants.
\end{proposition}

\begin{proof}
We only check that $AP(\mathcal{U})$ is closed in $C_{\infty }(A)$. Let $%
\{f^{(n)}\}$ be a sequence in $AP(\mathcal{U})$ converging to $f$
for the
uniform convergence topology, $\Phi \in \mathcal{L}$ , $\tau\in S(\mathcal{U})$%
, and $\{a_{n}\}$ a sequence in $A$. In order to prove that $f\in
AP(\mathcal{U})$, we must show that $\{(f\circ \tau\circ
\Phi_{a_{n}})\}$ has a convergent subsequence. Let us denote
$g^{(n)}=(f^{(n)}\circ \tau\circ \Phi )$, $g=(f\circ \tau\circ
\Phi )$ and $g^{(n)}_{a}=(f^{(n)}\circ \tau\circ \Phi_{a})$, to
simplify the notation. Thus, $g^{(1)}_{a_n}=(f^{(1)}\circ
\tau\circ \Phi_{a_n})$

Given $\epsilon >0$, there is $n_{1}<\omega $ such that $\parallel
g-g^{(n)}\parallel <\frac{\epsilon }{3}$ for $n>n_{1}$, where
$\parallel .\parallel $ denotes the supremum norm. The sequence
$\{g_{a_{n}}^{(1)}\}$ will contain a convergent subsequence , say
$\{g_{a_{(n,1)}}^{(1)}\}$. There is no loss of generality in
assuming that $\parallel
g_{a_{(n,1)}}^{(1)}-g_{a_{(m,1)}}^{(1)}\parallel <\frac{1}{2}$ for all $%
n,m<\omega $. We now define inductively a collection of convergent
sequences $\{g_{a_{(n,j)}}^{(j)}\}$, $j<\omega $, satisfying:

\begin{itemize}
\item $\{g_{a_{(n,j+1)}}^{(j+1)}\}$ is a subsequence of $%
\{g_{a_{(n,j)}}^{(j)}\}$;

\item $\parallel g_{a_{(n,j)}}^{(j)}-g_{a_{(m,j)}}^{(j)}\parallel <\frac{1}{%
2^{j}}$ for all $n,m<\omega $.
\end{itemize}
Take the diagonal subsequence $\{g_{a_{(n,n)}}\}$ of
$\{g_{a_{n}}\}$, and
let $n_{0}<\omega $ such that $n_{0}\geq n_{1}$ and $\frac{1}{2^{n_{0}}}<%
\frac{\epsilon }{3}$. For any $n,m\geq n_{0}$, we have
\begin{equation*}
\begin{array}{l}
\parallel g_{a_{(n,n)}}-g_{a_{(m,m)}}\parallel \leq \\
\parallel g_{a_{(n,n)}}-g_{a_{(n,n)}}^{n_{0}}\parallel +\parallel
g_{a_{(n,n)}}^{n_{0}}-g_{a_{(m,m)}}^{n_{0}}\parallel +\parallel
g_{a_{(m,m)}}^{n_{0}}-g_{a_{(m,m)}}\parallel \leq 3\cdot
\frac{\epsilon }{3}.
\end{array}
\end{equation*}
Hence, $\{g_{a_{(n,n)}}\}$ is a Cauchy sequence in $C_{\infty
}(A)$ what completes the proof.
\end{proof}

We have just shown that $AP(\mathcal{U})$ is a commutative Banach
algebra of continuous functions on $A$ (the base space of
$\mathcal{U}$) with the supremum norm. Therefore, we can state the
following.

\begin{definition}
\label{def4} The compact Gelfand space associated to the
commutative Banach algebra $AP(\mathcal{U})$ is denoted by
$b\mathcal{U}$.
\end{definition}
Next result is our characterization of the Bohr compactification.
The proof is split in several lemmas and propositions.
\begin{theorem} \label{th01}
The space $b\mathcal{U}$ is a realization of the Bohr
compactification of the topological structure $(\mathcal{U},\tau
)$.
\end{theorem}
According to the definition of the Bohr compactification of a
general topological structure, in order to prove the result above,
the following facts need be established:
\begin{itemize}
\item[B1] there is a map $\delta$ from $A$ into $b\mathcal{U}$
such that $\delta(A)$ is dense in $b\mathcal{U}$.

\item[B2] there is an algebraic structure on $b\mathcal{U}$
compatible with $\mathcal{U}$.

\item[B3] $b\mathcal{U}$ is the maximal compactification of
$\mathcal{U}$.
\end{itemize}
\medskip

 The verification of these properties requires some previous work.
Firstly, we state the following proposition on compact structures.

\begin{proposition}
\label{prop2}For each compact structure
$\mathcal{U}=(X,\mathcal{L})$ it holds that $C(X)$ is contained in
$AP(\mathcal{U})$.
\end{proposition}

\begin{proof}
For any $f\in C(X)$, let $\Phi $ be an arbitrary but fixed element
of $\mathcal{L}$, $\tau\in S(\mathcal{U})$, and let $n$ be the
arity of $\Phi $. For any index $j$ such that $1\leq j\leq n$,
define the map $\gamma :X_{j}\rightarrow C_{\infty}(\prod \{X_{i}
: 1\leq i\leq n\hbox{,}\ i\not= j\})$ by $\gamma (a)=f_{a}$, $a\in
X_{j}$, where $X_{i}=X$ for $1\leq j\leq n$ and $f_{a}=
(f\circ \tau\circ\Phi_{a})$. 
Taking into account that the topology of uniform convergence on
compact sets is proper (that is to say, the continuity of any map
$\varphi:X\times X\longrightarrow C$ yields automatically the
continuity of the map defined by $x\mapsto \varphi(x,.)$), it
follows that $\gamma $ is continuous on $X_{j}$ (see
\cite{engel}).

Since $X$ is compact, we obtain the compactness of $\gamma
(X_{j})$ in $C_{\infty }(\prod \{X_{i} : 1\leq i\leq n\hbox{,}\
i\not= j\})$. This completes the proof.
\end{proof}
Next lemma verifies item B1 above in the definition of the Bohr
compactification.

\begin{lemma}\label{le01}
There is a continuous map $\delta$ sending $A$ into a dense subset
of $b\mathcal{U}$.
\end{lemma}
\begin{proof} Since $b\mathcal{U}$ is the structure space of
$AP(\mathcal{U})$, it follows that $b\mathcal{U}$ is the set of
all multiplicative complex functionals on $AP(\mathcal{U})$
equipped with topology of pointwise convergence on
$AP(\mathcal{U})$ (cf. \cite[Appendix D]{rudin_fa}). Therefore,
the evaluation mapping $\delta:A \longrightarrow b\mathcal{U}$
defined by $\delta(a)[f]=f(a)$ for all $f\in AP(\mathcal{U})$
sends $A$ into $b\mathcal{U}$. Moreover, given that
$AP(\mathcal{U})\subset C(A)$, the map $\delta$ is clearly
continuous.

The algebra $AP(\mathcal{U})$ may be identified to a subalgebra of
$C(b\mathcal{U})$ by means of the Gelfand transform in the
following way: for each $f\in AP(\mathcal{U})$, its Gelfand
transform is a map $\widehat{f}\in C(b\mathcal{U})$ defined by
$\widehat{f}(p)=p(f)$ for all $p\in b\mathcal U$. Moreover, for
each $f\in AP(b\mathcal{U})$, it holds that
$\|\widehat{f}\|_{\infty} \leq \|f\|_{\infty}$ (here, $\| \
\|_{\infty}$ denotes the suppremum norm on either space,
$b\mathcal{U}$ and $A$). On the other hand, $AP(\mathcal U)$ is a
Banach algebra of continuous functions on $A$ equipped with the
suppremum norm. Since the evaluation mapping $\delta$ sends $A$
into $b\mathcal U$, it follows that $\|f\|_{\infty} \leq
\|\widehat{f}\|_{\infty}$ for all $f\in AP(\mathcal U)$. That is
to say, the Gelfand transform, in this case, is an isometry that
sends $AP(\mathcal U)$ isometrically into $C(b\mathcal U)$. In
fact, the Stone-Weierstrass theorem shows that the Gelfand
transform is an isometry onto $C(b\mathcal U)$. Now, observe that
$\delta(A)$ separates this algebra since, for each
$\widehat{f}\hbox{,} \widehat{g}\in C(b\mathcal U)$ with
$\widehat{f}\not= \widehat{g}$, there is $a\in A$ such that
$\widehat{f}(\delta(a))\not= \widehat{g}(\delta(a))$. This
property yields the density of the subspace $\delta(A)$ in
$b\mathcal U$.
\end{proof}
We have just verified that $b\mathcal U$ is a compactification of
$A$. In order to verify that $b\mathcal{U}$ coincides with the
Bohr compactification of $\mathcal{U}$, the next step is to equip
$\delta(A)$ with an algebraic structure $\delta(\mathcal U)$ for
$\mathcal L$ such that $\delta$ is a continuous homomorphism.
\begin{lemma} \label{le02}
The space $\delta(\mathcal A)$ may be provided with an algebraic
structure, $\delta(\mathcal U)$, for $\mathcal L$ such that the
map $\delta:A\longrightarrow \delta(A)$ is a continuous
homomorphism.
\end{lemma}
\begin{proof}
Let $\Phi$ be any function in $\mathcal L$ with arity $n$. We may
assume wlog that $n>1$ since the case $n=1$ is easier to deal
with. For $(a_{1},...,a_{n})\in A^{n}$, we set
$\Phi(\delta(a_{1}),...,\delta(a_{n}))=\delta(\Phi(a_{1},...,a_{n}))$.
We must check that $\Phi$ is properly defined on $\delta(A)$. Take
$(b_1,...b_n)\in A^{n}$ with $\delta(b_i)=\delta(a_i)$ for $1\leq
i\leq n$, and let us see that
$\delta(\Phi(a_{1},...,a_{n}))=\delta(\Phi(b_{1},...,b_{n}))$.
Observe that it suffices to verify that
$\delta(\Phi(a_{1},a_2...,a_{n}))=\delta(\Phi(b_{1},a_2...,a_{n}))$.
Otherwise, we split the procedure in several iterates affecting
one single coordinate each one of them. Now, since $\delta$ is the
evaluation mapping, we only need to show that
$f(\Phi(a_{1},a_2...,a_{n}))=f(\Phi(b_{1},a_2...,a_{n}))$ for all
$f\in AP(\mathcal U)$. Let $f$ be any fixed element in
$AP(\mathcal U)$ and consider the simple translation
$t=\Phi_{(a_2,...a_n)}$. We have that $f\circ t$ is in
$AP(\mathcal U)$ and $\delta(a_1)=\delta(b_1)$. Hence, $(f\circ
t)(a_1)=(f\circ t)(b_1)$. It follows that $\Phi$ is well defined
on $\delta(A)^n$. Moreover, the definition yields automatically
that $\delta$ is a continuous ($\Phi$)-homomorphism. This
completes the proof.
\end{proof}
We now verify that item B2 stated above holds. It will suffice to
prove that every operation $\Phi \in \mathcal{L}(\mathcal{U)}$
extends to a continuous operation $\Phi ^{b}\in
\mathcal{L}(b\mathcal{U})$. That is, if $\Phi $ has arity $n$, we
must extend $\Phi $, defined on $A^n$, to a continuous map $\Phi
^{b}$ defined on $(b\mathcal{U)}^{n}$. At this point, there is no
loss of generality in assuming that $A$ is a subspace of
$b\mathcal U$ since, otherwise, we would replace $A$ by
$\delta(A)$ and $\mathcal U$ by $\delta(\mathcal U)$. The first
step in this direction is a lemma concerning the extension of
mappings defined on products. We refer to \cite{engel} for this
result. Since we will need (and extend) it in the following we
include a sketch of the proof here. In the sequel, if $X$ is a
topological space and $\mu Z$ is a uniform space, we denote by
$C_{\infty }(X,\mu Z)$ the space of all continuous functions from
$X$ to $\mu Z$ equipped with the topology of uniform convergence.
That is to say, for $g\in C_{\infty }(X,\mu Z)$ a basic
neighbourhood base consists of sets of the form $N(g,B)=\{h\in
C_{\infty }(X,\mu Z):(g(x),h(x))\in B,\ x\in X,\ B\in \mu \}$.
Observe that this topology is generated by the uniformity
$\mu(\infty)$ which has a base consisting of the sets
$B^{+}=\{(f,g)\in (C_{\infty }(X,\mu Z))^{2}: (f(x),g(x))\in B,
x\in X\}$, with $B\in \mu$.

\begin{lemma}
\label{lem1}Let $X$, $Y$ and $K$ be topological spaces with $K$
being a compactification of $X$ , and let $\mu Z$ be a uniform
space. Let $f:X\times Y\rightarrow \mu Z$ be a continuous map such
that, for all $y\in Y$, the canonical map $f(.,y):X\rightarrow \mu
Z$ admits a continuous extension $\overline{f}(.,y):K\rightarrow
\mu Z$. Consider the following properties:

\begin{itemize}
\item[(a)] The family $\{f(x,.):x\in X\}$ is relatively compact in
$C_{\infty }(Y,\mu Z)$;

\item[(b)] The family $\{f(x,.):x\in X\}$ is equicontinuous;

\item[(c)] The map $\gamma :Y\rightarrow C_{\infty }(X,\mu Z)$
defined by $\gamma(y)(x)=f(x,y)$ ($x\in X$) is continuous;

\item[(d)] f extends continuously to $K\times Y$.
\end{itemize}
Then we have: $(a)\Rightarrow (b)\Rightarrow (c) \Rightarrow (d)$.
\end{lemma}

\begin{proof}
$(a\Longrightarrow b)$ Let $B$ be an arbitrary vicinity in $\mu $. The family $%
\{N(g,B):g\in C(Y,\mu Z)\}$ is an open cover of $C_{\infty }(Y,\mu
Z)$. Since $C=\{f(x,.):x\in X\}$ is relatively compact, it follows
there is a finite subfamily, say $\{N(g_{l},B):1\leq l\leq m\}$,
that covers $C$. Now,
for each $y\in Y$ there is a neighbourhood $V_{(y,l)}$ such that $%
(g_{l}(y),g_{l}(y^{\prime }))\in B$ for all $y^{\prime }\in
V_{(y,l)}$. Take
$V=\cap _{l=1}^{m}V_{(y,l)}$. It is easily verified that $%
(f(x,y),f(x,y^{\prime }))\in B^{2}$ for all $y^{\prime }\in V$ and
$x\in X$. This proves the equicontinuity of $C.$

$(b\Longrightarrow c)$ Assuming that $\{f(x,.):x\in X\}$ is
equicontinuous. Given $B\in \mu $ and $y_{0}\in Y$, there is a
neighbourhood $V$ of $y_{0}$ such that $(f(x,y_{0}),f(x,y))\in B$
for all $y\in V$. That is, $\gamma (V)\subset N(\gamma(y_0),B)$
and this yields the continuity of $\gamma $.

$(c\Longrightarrow d)$ Let $\overline{f}(.,y):K\rightarrow \mu Z$
be the continuous extension of $f(.,y)$ for each $y\in Y$. In
order to prove the continuity of $\overline{f}:K\times
Y\rightarrow \mu Z$, it suffices to show that, for every $p\in K$,
the mapping $\overline{f}:(X\cup \{p\})\times Y\rightarrow \mu Z$
is continuous. Now, given $B\in \mu $, take and arbitrary point
$y\in Y$. Since $\overline{f}(.,y)$ is continuous, there is a
neighbourhood $U_{(p,y)}$ of $p$ such that for all $x\in
U_{(p,y)}$ it holds $(\overline{f}(p,y),f(x,y))\in B$. Since
$\gamma :Y\rightarrow C_{\infty }(X,\mu Z)$ is continuous, there
is a neighbourhood $V_{y}$ such
that $(f(x,y),f(x,y^{\prime }))\in B$ for all $y^{\prime }\in V_{y}$ and $%
x\in X$. Hence, for each $(x,y^{\prime })\in U_{(p,y)}\times
V_{y}$ we have
that $(f(p,y),f(x,y^{\prime }))\in B^{2}$. This proves the continuity of $%
\overline{f}$.
\end{proof}

Next proposition is one of the main results of this note. It
extends the lemma above to finite products.

\begin{proposition}
\label{prop3}Let $\{X_{i}:1\leq i\leq n\}$ be a finite family of
topological spaces and let $K_{i}$ be a fixed compactification of
$X_{i}$ for $1\leq i\leq n$. If $\mu Z$ is a complete uniform
space and $g:\prod\{X_i: 1\leq i\leq n\}\rightarrow \mu Z$ is a
continuous mapping satisfying:

\begin{itemize}
\item[(a)] $g_{\overrightarrow{x}}$ \vspace{0.05in} can be
extended to a continuous map
$\overline{g}_{\overrightarrow{x}}:K_{j}\rightarrow \mu Z$, for
all $\overrightarrow{x}\in \prod\{X_i: 1\leq i\leq n \hbox{,}\
i\not= j\}$, $1\leq j\leq n$; \vspace{0.05in} and

\item[(b)] $\{g_{x_{j}}:x_{j}\in X_{j}\}$ \vspace{0.05in} is
relatively compact in\\
$C_{\infty}(\prod\{X_i: 1\leq i\leq n \hbox{,}\ i\not=j\},\mu Z)$,
$1\leq j\leq n$.
\end{itemize}

Then $g$ can be extended to a continuous map
$\overline{g}:\prod\{K_i: 1\leq i\leq n\}\rightarrow \mu Z$.
\end{proposition}
\begin{proof}
In order to simplify the notation, we treat the case $n=3$ only,
as this is representative for the general case. The proof for
$n>3$ is obtained proceeding by induction.

It is clear from Lemma \ref{lem1} that $g$ extends with continuity
to $\overline{g}_1:K_1\times X_2\times X_3\longrightarrow \mu Z$.
Now, we prove that $g$ also extends with continuity to
$\overline{g}_2:K_1\times K_2\times X_3\longrightarrow \mu Z$.
This will suffice since the same arguments apply to extend $g$ to
the whole product of compact spaces.

Denote by $\overline{g}_{x_{2}}$ to the continuous extension to
$K_1\times X_3 $ of $g_{x_{2}}$ for each $x_{2}\in X_{2}$.

(i) \ Firstly, we prove that $\{\overline{g}_{x_{2}}:x_{2}\in
X_{2}\}$ is relatively compact in \vspace{0.05in} $C_{\infty
}(K_1\times X_3,\mu Z)$.

Indeed, the restriction mapping $r:C_{\infty}(K_1\times X_3,\mu
Z)\longrightarrow C_{\infty}(X_1\times X_3,\mu Z)$ is one-to-one
and uniformly continuous. By Lemma \ref{lem1}, each map $f$ in
$C_{\infty}(X_1\times X_3,\mu Z)$ may be extended to a continuous
map $\overline{f}$ in $C_{\infty}(K_1\times X_3,\mu Z)$. It is
easily verified that for all $B\in \mu$, if $(f,g)\in B^{+}$ then
$(\overline{f},\overline{g})\in \overline{B}^{\ +}\subset (B\circ
B)^{+}$ (here $\overline{B}^{\ +}$ denotes the closure of $B$ in
$\mu Z\times \mu Z$). Thus, $r$ is an onto uniform isomorphism.
This \vspace{0.05in} yields (i).

(ii) \ For each $p_{0}\in K_1$ and $\{x_{\delta}:\delta \in D\}$ a
net in $X_1$ converging to $p_{0}$, it holds that $\{g_{(x_{\delta
},a)}:\delta \in D\}$ converges to $\overline{g}_{(p_{0},a)}$ in
$C_{\infty }(X_{2},\mu Z)$ for all $a\in X_3$.

Indeed, since $\{\overline{g}_{x_{2}}:x_{2}\in X_{2}\}$ is
relatively compact in $C_{\infty }(K_1\times X_3,\mu Z)$, given a
vicinity $B\in \mu $, there is a finite subset $\{b_{j}:1\leq
j\leq l\}\subset X_{2}$ such that for each fixed $b\in X_{2}$
there is an $b_{j}$ with
\begin{equation}
(\overline{g}(p,b,a),\overline{g}(p,b_{j},a))\in B \tag{I}
\end{equation}
\vspace{0.05in} for all $p\in K_1$ and $a\in X_3$.

On the other hand, the map $g_{(b_{j},a)}$ can be extended with
continuity to $K_1$ for all $(b_{j},a)$, $1\leq j\leq l$.
Therefore for each fixed $a\in X_3$ and $j$ with $1\leq j\leq l$,
there is $\delta _{j}\in D$ such that, for $\delta \geq \delta
_{j}$, it holds
\begin{equation}
(\overline{g}(p_{0},b_{j},a),\overline{g} (x_{\delta
},b_{j},a))\in B\text{. } \tag{II}
\end{equation}

Let $\delta _{0}\in D$ be such that $\delta _{0}\geq \delta _{j}$
for $1\leq j\leq l$. Applying (I), for every $x_{2}\in X_{2}$
there is an element $b_{j}$ such that
\begin{equation*}
(\overline{g}(p_{0},x_{2},a), \overline{g}(p_{0},b_{j},a))\in B
\end{equation*}
and
\begin{equation*}
(\overline{g}(x_{\delta }, b_{j},a),\overline{g}(x_{\delta
},x_{2}, a))\in B
\end{equation*}

for all $\delta \in D$. Appying also (II), we obtain
\begin{equation*}
(\overline{g}(p_{0},x_{2},a), \overline{g}(x_{\delta
},x_{2},a))\in B^{3}
\end{equation*}

for all $x_{2}\in X_{2}$ and $\delta \geq \delta _{0}$. This
proves (ii).

Since $g_{(x_{\delta },a)}$ can be extended with continuity to
$\overline{g}_{(x_{\delta }, a)}:K_{2}\rightarrow \mu Z$ for all
$\delta \in D$ and $a\in X_3$, it follows from (ii) that
$\overline{g}_{(p,a)}$ can also be extended with continuity to
$K_{2}$ for all $p\in K_2$ and $a\in X_3$. The latter property
with (i) and Lemma \ref{lem1} implies that $g$ can be extended
with continuity to $K_1\times K_2\times X_3$. This completes the
proof.
\end{proof}

We are now ready to establish item B2.

\begin{proposition}\label{pro1}
Every $\Phi \in \mathcal{L}(\mathcal{U}$ $)$, of arity $n$, can be
extended to a continuous map $\Phi ^{b}$ from $(b\mathcal{U})^{n}$
into $b\mathcal{U}$.
\end{proposition}

\begin{proof}
Let $f$ be an arbitrary element of $AP(\mathcal{U})$ and let
$g=f\circ \Phi $. By the definition of $AP(\mathcal{U})$, we know
that for all $\overrightarrow{a}\in \prod \{ A_i: 1\leq i\leq n,\
i\not= j\}$, the map $g_{\overrightarrow{a}}=f\circ
t^{\Phi}_{\overrightarrow{a}}$ belongs to $AP(\mathcal{U})$ (here
$A_{i}=A$, $1\leq i\leq n$). As a consequence,
$g_{\overrightarrow{a}}$ can be extended with continuity to
$b\mathcal{U}$. We also have that the set $\{g_{a_{j}}:a_{j}\in
A_{j}\}$ is relatively compact in $C_{\infty }(\prod \{ A_i: 1\leq
i\leq n,\ i\not= j\})$. Applying Lemma \ref{prop3}, we obtain that
$g$ can be extended to a continuous function
$\overline{g}:(b\mathcal{U})^{n}\rightarrow \mathbb{C}$. Now, for
each $\overrightarrow{p}\in (b\mathcal{U})^{n}$ we define $\Phi
^{b}(\overrightarrow{p})$ to hold the equality $(\overline{f}\circ
\Phi ^{b}) (\overrightarrow{p})=\overline{g}(\overrightarrow{p})$,
for all $f\in AP(\mathcal{U})$, where $\overline{f}$ denotes the
continuous extension of $f$ to $b\mathcal{U}$. Clearly, the map
$\Phi ^{b}$ is well defined, continuous and extends $\Phi$, since
$AP(\mathcal{U})$ can be identified to the space of all continuous
complex-valued functions on $b\mathcal{U}$. This completes the
proof.
\end{proof}
Finally, we deal with item B3 in next proposition.
\begin{proposition}\label{pro2}
The compact structure $b\mathcal U$ is the maximal compatible
compactification of $\mathcal U$.
\end{proposition}
\begin{proof}
Let $(K,\rho)$ be a compatible compactification of $\mathcal U$
where $\rho:A\longrightarrow K$ is the canonical injection.
Observe that, since $\rho$ is a homomorphism, we have that $f\circ
\rho\in AP(\mathcal U)$ for all $f\in C(K)$. We define the map
$\Gamma:b\mathcal{U}\longrightarrow K$ to hold the equality
$f(\Gamma(p))=p(f\circ \rho)$ for all $f\in C(K)$. The map
$\Gamma$ is clearly continuous by definition. Moreover, if $a$ is
an arbitrary element of $A$, we have that $f(\Gamma\circ
\delta)(a))=f(\Gamma(\delta(a)))=\delta(a)(f\circ \rho)=(f\circ
\rho)(a)=f(\rho(a))$, for all $f\in C(K)$. As a consequence,
$(\Gamma\circ \delta)(a)= \rho(a)$ for all $a\in A$. This proves
that $(K,\rho)\leq_{\Gamma} (b\mathcal U,\delta)$. Hence, the
maximallity of $b\mathcal U$ has been verified.
\end{proof}
To finish this section we set the proof of Theorem \ref{th01}.
\begin{proof}[\sc{Proof of Theorem \ref{th01}:}]
It suffices to combine Lemma \ref{le01}, Proposition \ref{pro1}
and Proposition \ref{pro2}.
\end{proof}
\section{Isometries}
In this section we apply the results obtained in the previous part
in order to represent linear isometries defined between spaces of
almost periodic functions. Here on, $\mathcal{L}$ denotes a set of
symbols of constants and functions and $\mathcal{U}$ and
$\mathcal{V}$ are two topological structures for $\mathcal{L}$
whose domains are the spaces $A$ and $B$, respectively. There is
no loss of generality in assuming that the sets of almost periodic
functions $AP(\mathcal{U})$ and $AP(\mathcal{V})$ separate the
points of $A$ and $B$, respectively (otherwise, we should take the
canonical quotient space, obtained by identifying the points which
may not be separated by almost periodic functions). The sets $A$
and $B$ inherite a topology as subspaces of their respective Bohr
compactifications, which is called the {\it Bohr topology}. These
topological spaces are denoted by $A^{\sharp}$ and $B^{\sharp}$.
The topological structures $\mathcal{U}$ and $\mathcal{V}$
equipped with the Bohr topologies are denoted by
$\mathcal{U}^{\sharp}$ and $\mathcal{U}^{\sharp}$, respectively.

If $\Phi$ is a $2$-ary function in $\mathcal{L}$, for every $y\in
B$, we denote by $\Phi_{y}$ (resp. $_{y}\Phi$) the map defined as
$\Phi_{y}(z)=\Phi(z,y)$ (resp. $_{y}\Phi(z)=\Phi(y,z))$ for all
$z\in B$. An isometry $T$ of $AP(\mathcal{U})$ onto
$AP(\mathcal{V})$ \textit{commutes with ($\Phi$) translations}
when there is a map $\lambda: B \longrightarrow A$ such that
$(Tf)\circ \Phi_{y}=T(f\circ \Phi_{\lambda(y)})$ for all $f\in
AP(\mathcal{V})$. We say that $T$ is {\it non-vanishing} when
$(Tf)(y)\not=0$ for all $y\in B$ if and only if $f(x)\not=0$ for
all $x\in A$.
\begin{theorem} \label{th31}
Let $T$ be a non-vanishing linear isometry of $AP(\mathcal{U})$
onto $AP(\mathcal{V})$.
\begin{itemize}
\item[(i)] There exists a continuous map $h$ of $B^{\sharp}$ onto
$A^{\sharp}$ and a continuous mapping $w:B^{\sharp}\longrightarrow
\Bbb C$, $|w|\equiv 1$, such that
$$
(Tf)(y)=w(y)\cdot f(h(y))\text{ for all } y\in B \text{ and all
}f\in AP(\mathcal{U}).
$$
\item[(ii)] If $\Phi \in \mathcal{L}$ is  2-ary, associative, has
a neutral element $1$ and $T$ commutes with translations, then
there is a singleton $b\in B$ such that $h\circ \ _{b}{\Phi}$ is a
$\Phi$-isomorphism of \ $\mathcal{V}$ onto $\mathcal{U}$.

\end{itemize}
\end{theorem}
\begin{proof}
The statement (i) is consequence of the Banach-Stone theorem and
the fact that $T$ preserves non-vanishing functions (see
\cite{iso_ap} to find an analogous result for topological groups).
Thus, only (ii) need to be checked. There is no loss of generality
in assuming that $w=1$ (otherwise, we should replace $T$ by the
isometry $\tilde{T}(f)=w^{-1}\cdot T(f)$ \ ). Denote by
$1_{\mathcal{U}}$ the neutral element in $\mathcal{U}$ and let $b$
be equal to $h^{-1}(1_{\mathcal{U}})$. We define the isometry $R$
of $AP(\mathcal{U})$ onto $AP(\mathcal{V})$ by
$Rf(y)=Tf(\Phi(b,y))$ for all $y\in B$. It is easy to check that
$R$ also commute with the $\Phi$-translations. Indeed,

\begin{equation*}
\begin{array}{ll}
((Rf)\circ \Phi_{z})(y)=(Rf)(\Phi(y,z))=
(Tf)(\Phi(b,\Phi(y,z)) &=\\
(Tf)(\Phi(\Phi(b,y),z)) \ \ \ \ = \ \ \ \ ((Tf)\circ~\Phi_{z})(\Phi(b,y)) \ &=\\
T(f\circ \Phi_{\lambda(z)})(\Phi(b,y)) \ \ = \ \ \ \ R(f\circ
\Phi_{\lambda(z)})(y).
\end{array}
\end{equation*}
Applying (i) to $R$ we obtain a map $k:\mathcal{V} \longrightarrow
\mathcal{U}$ such that $Rf=f\circ k$ for all $f\in
AP(\mathcal{U})$. Since $AP(\mathcal{V})$ separates the points of
$B$, it is readily seen that $k(y)=h(\Phi(b,y))$ for all $y\in B$.
Thus,
$k(1_{\mathcal{V}})=h(\Phi(b,1_{\mathcal{V}}))=h(b)=1_{\mathcal{U}}$.
On the other hand

\noindent $f(k(y))=Rf(y)=Rf(\Phi(1_{\mathcal{V}}))=((Rf)\circ
\Phi_{y})(1_{\mathcal{V}})= R(f\circ
\Phi_{\lambda(y)})(1_{\mathcal{V}})=(f\circ \Phi_{\lambda
(y)})(k(1_{\mathcal{V}}))= (f\circ \Phi_{\lambda
(y)})(1_{\mathcal{U}})=f(\Phi(1_{\mathcal{U}},\lambda (y))=
f(\lambda (y))$.

Again, this means that $k(y)=\lambda (y)$ for all $y\in B$.
Finally

\noindent $f(k(\Phi(z,y)))=Rf(\Phi(z,y))=((Rf)\circ
\Phi_{y})(z)=R(f\circ \Phi_{k(y)})(z)= (f\circ
\Phi_{k(y)})(k(z))=f(\Phi(k(z),k(y))$.

That is to say, $k(\Phi(z,y))=\Phi(k(z),k(y))$ for all $z,y \in
B$. This proves that $k$ is a $\Phi$-isomorphism.
\end{proof}
{\bf Acknowledgement}: We would like to thank the referee of an
original draft of this paper for a number of very helpful
comments.

\end{document}